\documentclass{amsart}
\usepackage{amsmath,amssymb,bm,amscd}
\usepackage{color}
\usepackage{amsmath}
\usepackage{amssymb}
\usepackage{amscd}
\usepackage{amsthm}
\usepackage[dvips]{graphicx}

\theoremstyle{definition}
\newtheorem{defn}{Definition}[section]

\newtheorem{rem}[defn]{Remark}
\theoremstyle{plain}
\newtheorem{thm}[defn]{Theorem}
\newtheorem{prop}[defn]{Proposition}
\newtheorem{lem}[defn]{Lemma}
\newtheorem{cor}[defn]{Corollary}

\newtheorem{question}[defn]{Question}
\newcommand{\KH}{\operatorname{KH}}
\numberwithin{equation}{section}
\title[]{The Rasmussen invariant, four-genus and three-genus of an almost positive knot are equal}
\author{Keiji Tagami}
\date{\today}
\address{
Department of Mathematics,
Tokyo Institute of Technology,
Oh-okayama, Meguro, Tokyo 152-8551, Japan
}
\email{tagami.k.aa@m.titech.ac.jp}
\begin{document}
\maketitle
\begin{abstract}
An oriented link is positive if it has a link diagram whose crossings are all positive. 
An oriented link is almost positive if it is not positive and has a link diagram with exactly one negative crossing. 
It is known that the Rasmussen invariant, $4$-genus and $3$-genus of a positive knot are equal. 
In this paper, we prove that the Rasmussen invariant, $4$-genus and $3$-genus of an almost positive knot are equal. 
Moreover, we determine the Rasmussen invariant of an almost positive knot in terms of its almost positive knot diagram. 
As corollaries, we prove that any almost positive knot is not homogeneous, and there is no almost positive knot of $4$-genus one. 
\end{abstract}
\section{Introduction}
An oriented link is {\it positive} if it has a link diagram whose crossings are all positive. 
For positive links, there are many studies. 
For example, Rudolph \cite{positive_strong} and Nakamura \cite[Lemma~$4.1$]{nakamura1} proved that every positive link is strongly quasipositive, 
Cromwell \cite[Corollary~$2.1$]{homogeneous} proved that a positive link has positive Conway polynomial (that is, all the coefficients of the polynomial are not negative) 
and Przytycki \cite[Theorem~$1$]{negative_signature} proved that all nontrivial positive links have negative signatures. 
\par
An oriented link is {\it almost positive} if it is not positive and has a link diagram with exactly one negative crossing. 
Such a diagram is called {\it almost positive diagram}. 
It is known that almost positive links have many properties similar to those of positive links. 
For instance, Cromwell \cite[Corollary~$2.2$]{homogeneous} also proved that any almost positive link has positive Conway polynomial, 
while Przytycki and Taniyama \cite[Corollary~$1.7$]{almost_negative_signatire} proved that almost positive links have negative signatures. 
Moreover, many examples of almost positive links are strongly quasipositive (we do not know whether all almost positive links are strongly quasipositive). 
\par
In \cite{rasmussen1}, Rasmussen introduced a knot invariant, called the Rasmussen invariant, which gives a lower bound of the $4$-genus. 
He proved that the Rasmussen invariant, $4$-genus and $3$-genus are equal for a positive knot \cite[Theorem~$4$]{rasmussen1}. 
\par
In this paper, we prove the following theorem: 
\begin{thm}\label{main}
Let $K$ be an almost positive knot. Then we obtain 
\begin{align*}
s(K)=2g_{4}(K)=2g_{3}(K), 
\end{align*}
where $s(K)$, $g_{4}(K)$ and $g_{3}(K)$ are the Rasmussen invariant, $4$-genus and $3$-genus of $K$, respectively. 
\end{thm}
Moreover, we compute the Rasmussen invariant of a knot represented by an almost positive diagram from the diagram (Theorem~$\ref{main2}$). 
%
%
%
%
%
From Theorem~$\ref{main}$, we notice that the Rasmussen invariant of an almost positive knot is positive (which is an analogous property to the signature by Przytycki and Taniyama \cite{almost_negative_signatire}). 
\par
As corollaries of Theorem~$\ref{main}$, we obtain the following results (Corollaries~$\ref{cor1}$ and $\ref{cor2}$): 
\begin{cor}\label{cor1}
Any almost positive knot is not homogeneous. 
\end{cor}
%
%
%
%
\begin{cor}[negative answer to Question~$7.1$ in \cite{stoimenow2}]\label{cor2}
There is no almost positive knot of $4$-genus (or unknotting number) one. 
\end{cor}
The proofs of Theorem~$\ref{main}$, and Corollaries~$\ref{cor1}$ and $\ref{cor2}$ are given in Section~$\ref{section_main}$. 
\par
%
Homogeneous links are introduced by Cromwell \cite{homogeneous} (see also \cite{baader}, \cite{banks1} and \cite{homogeneous-seifert}). 
From its definition, any positive knot is homogeneous. 
The Rasmussen invariants of homogeneous knots are determined by Abe \cite{abe2} in terms of their diagrams. 
Theorems~$\ref{main}$ and $\ref{main2}$, and Corollary~$\ref{cor1}$ give us a new class of non-homogeneous knots whose Rasmussen invariants are well understood. 
Furthermore, it immediately follows from Corollary~$\ref{cor1}$ that the homogeneity of knots is a serious difference between positive knots and almost positive knots. 
%
\par
Let $B_{n}$ be the $n$-string braid group with the canonical generators $\{\sigma_{i} \}_{i=1}^{n-1}$. 
A {\it positive band} is any conjugate $w\sigma_{i} w^{-1}$ ($w\in B_{n}$, $1\leq i\leq n-1$) and a {\it positive embedded band} is one of the positive bands $\sigma_{i,j}:=(\sigma_{i}\cdots \sigma_{j-2})\sigma_{j-1}(\sigma_{i}\cdots \sigma_{j-2})^{-1}$ ($1\leq i<j\leq n-1$). 
A {\it (strongly) quasipositive braid} is a product of some positive (embedded) bands and a {\it (strongly) quasipositive link} is an oriented link realized by the closure of a (strongly) quasipositive braid. 
%
Shumakovitch \cite[Proposition~$1$.F]{shumakovitch2} proved that for a strongly quasipositive knot $K$, we obtain $s(K)=2g_{4}(K)=2g_{3}(K)$. 
Hence, Theorem~$\ref{main}$ is evidence towards an affirmative answer to the following question given by Stoimenow \cite{stoimenow1}: 
\begin{question}{\cite[Question~$4$]{stoimenow1}}
Is any almost positive link strongly quasipositive, or at least quasipositive?
\end{question}
\par
This paper is organized as follows: 
In Section~$\ref{def_kh}$, we recall the definition of Khovanov homology and give the key property of the Rasmussen invariant needed to prove our theorem. 
In Section~$\ref{section_main}$, we prove Theorem~$\ref{main}$, and Corollaries~$\ref{cor1}$ and $\ref{cor2}$. 
\section{Khovanov homology}\label{def_kh}
In this section, we recall the definition of (rational) Khovanov homology. 
Let $L$ be an oriented link. 
Take a diagram $D$ of $L$ and an ordering of the crossings of $D$. 
For each crossing of $D$, we define $0$-smoothing and $1$-smoothing as in Figure~$\ref{smoothing}$. 
A smoothing of $D$ is a diagram where each crossing of $D$ is changed to either its $0$-smoothing or $1$-smoothing. 
\begin{figure}[!h]
\begin{center}
\includegraphics[scale=0.5]{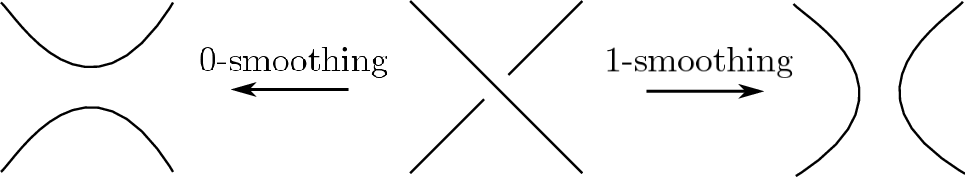}
\end{center}
\caption{0-smoothing and 1-smoothing. }
\label{smoothing}
\end{figure}
Let $n$ be the number of the crossings of $D$. Then $D$ has $2^{n}$ smoothings. 
By using the given ordering of the crossings of $D$, we have a natural bijection between the set of smoothings of $D$ and the set $\{0, 1\}^{n}$, where, to any $\varepsilon =(\varepsilon _{1}, \dots, \varepsilon _{n})\in \{0, 1\}^{n}$, we associate the smoothing $D_{\varepsilon }$ where the $i$-th crossing of $D$ is $\varepsilon_{i}$-smoothed. 
Each smoothing $D_{\varepsilon }$ is a collection of disjoint circles. 
\par
Let $V$ be a graded free $\mathbf{Q}$-module generated by $1$ and $X$ with $\operatorname{deg}(1)=1$ and $\operatorname{deg}(X)=-1$. 
Let $k_{\varepsilon }$ be the number of the circles of the smoothing $D_{\varepsilon }$. 
Put $M_{\varepsilon }=V^{\otimes k_{\varepsilon }}$. 
The module $M_{\varepsilon }$ has a graded module structure, that is, for $v=v_{1}\otimes\cdots\otimes v_{k_{\varepsilon }}\in M_{\varepsilon }$, $\deg(v):=\deg(v_{1})+\cdots+\deg(v_{k_{\varepsilon }})$. 
Then define 
\begin{align*}
C^{i}(D)&:=\bigoplus_{|\varepsilon |=i }M_{\varepsilon }\{i\},  
\end{align*}
where $|\varepsilon |=\sum_{i=1}^{m}\varepsilon _{i}$. 
Here, $M_{\varepsilon}\{i\}$ denotes $M_{\varepsilon}$ with its gradings shifted by $i$ (for a graded module $M=\bigoplus_{j\in\mathbf{Z}}M^{j}$ and an integer $i$, we define the graded module $M\{i\}=\bigoplus_{j\in\mathbf{Z}}M\{i\}^{j}$ by $M\{i\}^{j}=M^{j-i}$). 
\par
The differential map $d^{i}\colon C^{i}(D)\rightarrow C^{i+1}(D)$ is defined as follows. 
Fix an ordering of the circles for each smoothing $D_{\varepsilon }$ and associate the $i$-th tensor factor of $M_{\varepsilon }$ to the $i$-th circle of $D_{\varepsilon }$.  
Take elements $\varepsilon$ and $\varepsilon ' \in \{0, 1\}^{n}$ such that $\varepsilon _{j}=0$ and $\varepsilon' _{j}=1$ for some $j$ and that $\varepsilon _{i}=\varepsilon' _{i}$ for any $i\neq j$. 
For such a pair $(\varepsilon , \varepsilon ')$, we will define a map $d_{\varepsilon \rightarrow \varepsilon '}\colon M_{\varepsilon }\rightarrow M_{\varepsilon '}$. 
\par
In the case where two circles of $D_{\varepsilon }$ merge into one circle of $D_{\varepsilon' }$,  the map $d_{\varepsilon \rightarrow \varepsilon '}$ is the identity on all factors except the tensor factors corresponding to the merged circles where it is a multiplication map $m\colon V\otimes V\rightarrow V$ given by: 
\begin{center}
$m(1\otimes 1)=1$,\  $m(1\otimes X)=m(X\otimes 1)=X$,\  $m(X\otimes X)=0$. 
\end{center}
\par
In the case where one circle of $D_{\varepsilon }$ splits into two circles of $D_{\varepsilon' }$,  the map $d_{\varepsilon \rightarrow \varepsilon '}$ is the identity on all factors except the tensor factor corresponding to the split circle where it is a comultiplication map $\Delta \colon V\rightarrow V\otimes V$ given by:
\begin{center}
$\Delta (1)=1\otimes X+X\otimes 1$,\  $\Delta (X)=X\otimes X$. 
\end{center}
\par
If there exist distinct integers $i$ and $j$ such that $\varepsilon _{i}\neq\varepsilon '_{i}$ and that $\varepsilon _{j}\neq\varepsilon '_{j}$, then define $d_{\varepsilon \rightarrow \varepsilon '}=0$. 
\par
In this setting, we define a map $d^{i}\colon C^{i}(D)\rightarrow C^{i+1}(D)$ by $\sum_{|\varepsilon|=i}d_{\varepsilon}^{i}$, where $d_{\varepsilon}^{i}\colon M_{\varepsilon}\rightarrow C^{i+1}(D)$ is defined by 
\begin{align*}
d^{i}_{\varepsilon}(v):=\sum_{|\varepsilon'|=i+1 }(-1)^{l(\varepsilon, \varepsilon' )}d_{\varepsilon \rightarrow \varepsilon '}(v).   
\end{align*}
Here $v\in M_{\varepsilon }\subset C^{i}(D)$ and $l(\varepsilon, \varepsilon')$ is the number of $1$'s in front of (in our order) 
the factor of $\varepsilon$ which is different from $\varepsilon'$. 
\par
We can check that ($C^{i}(D)$, $d^{i}$) is a cochain complex and we denote its $i$-th homology group by $H^{i}(D)$. 
We call these the unnormalized Khovanov homology of $D$. 
Since the map $d^{i}$ preserves the grading of $C^{i}(D)$, the group $H^{i}(D)$ has a graded structure $H^{i}(D)=\bigoplus_{j\in\mathbf{Z}}H^{i,j}(D)$ induced by that of $C^{i}(D)$. 
For any link diagram $D$, we define its Khovanov homology $\KH^{i, j}(D)$ by 
\begin{center}
$\KH^{i, j}(D)=H^{i+n_{-}, j-n_{+}+2n_{-}}(D)$, 
\end{center}
where $n_{+}$ and $n_{-}$ are the number of the positive and negative crossings of $D$, respectively. 
The grading $i$ is called the homological degree and $j$ is called the $q$-grading. 
\begin{thm}[\cite{Bar-Natan-1}, \cite{khovanov1}]
Let $L$ be an oriented link and $D$ a diagram of $L$. Then $\KH(L):=\KH(D)$ is a link invariant. 
Moreover, the graded Euler characteristic of the homology $\KH(L)$ equals the Jones polynomial of $L$, that is, 
\begin{align*}
V_{L}(t)=(q+q^{-1})^{-1}\sum_{i, j\in\mathbf{Z}}(-1)^{i}q^{j}\dim_{\mathbf{Q}}{\KH^{i, j}(L)}\Big|_{q=-t^{\frac{1}{2}}}, 
\end{align*}
where $V_{L}(t)$ is the Jones polynomial of $L$.  
\end{thm}
The following is well known result and the key property to prove Theorem~$\ref{main}$. 
\begin{prop}\label{key}
Let $K$ be an oriented knot. Then we obtain $\KH^{0, s(K)\pm 1}(K)\neq 0$, where $s(K)$ is the Rasmussen invariant of $K$. 
\end{prop}
\begin{proof}
The $0$-th term of the Lee homology of $K$ is generated by two elements $v_{\max}$ and $v_{\min}$ whose q-gradings are $s(K)+1$ and $s(K)-1$, respectively. It is known that there is a spectral sequence whose $E_{1}$ term is the Khovanov homology of $K$ and $E_{\infty}$ term is the Lee homology. 
From the construction of the spectral sequence, there are nonzero elements $\tilde{v}_{\max}$ and $\tilde{v}_{\min}$ in $\KH^{0}(K)$ whose q-gradings are $s(K)+1$ and $s(K)-1$, respectively. 
\end{proof}
%
\begin{rem}
As an application of Proposition~$\ref{key}$, the author \cite[Corollary~$5.3$]{tagami2} computed the Rasmussen invariant of twisted Whitehead doubles (with sufficiently many twists) of any knot. 
\end{rem}
\section{Proof of Theorem~$\ref{main}$}\label{section_main}
In this section, we prove Theorem~$\ref{main}$. 
\par
An oriented link diagram is {\it almost positive} if it has exactly one negative crossing. 
First, we introduce a result of Stoimenow, which gives a method for computing the $3$-genus of a link represented by an almost positive diagram. 
\begin{thm}{\cite[Corollary~$5$ and the proof of Theorems~$5$ and $6$]{stoimenow1}}\label{stoimenow1}
Let $D$ be an almost positive diagram of a non-split link $L$ with a negative crossing $p$. 
Denote the genus of $L$ by $g_{3}(L)$ and the genus of the Seifert surface obtained from $D$ (by Seifert's algorithm) by $g_{3}(D)$. 
\begin{enumerate}
\item If there is no (positive) crossing joining the same two Seifert circles of $D$ as the two circles which are connected by the negative crossing $p$, we have $g_{3}(L)=g_{3}(D)$ (see the left of Figure~$\ref{fig:negative}$). 
\item If there is a (positive) crossing joining the same two Seifert circles of $D$ as the two circles which are connected by the negative crossing $p$, we have $g_{3}(L)=g_{3}(D)-1$ (see the right of Figure~$\ref{fig:negative}$). 
\end{enumerate}
\end{thm}
\begin{figure}[!h]
\begin{center}
\includegraphics[scale=0.55]{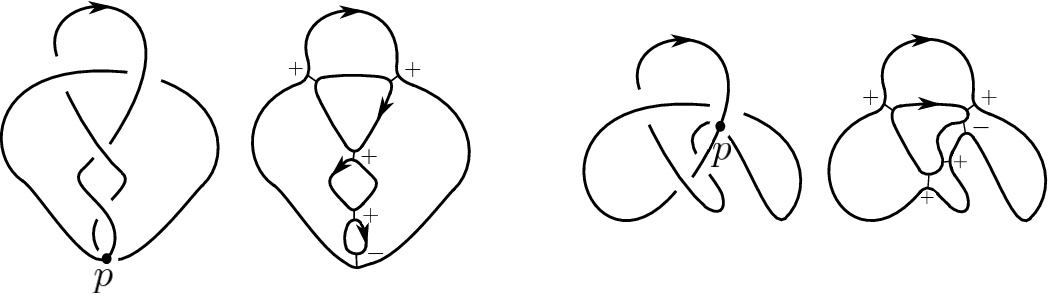}
\end{center}
\caption{In the left picture, there is no crossing joining the same two Seifert circles as the two circles which are connected by the negative crossing $p$. 
In the right picture, there is a crossing joining the same two Seifert circles as the two circles which are connected by the negative crossing $p$.}
\label{fig:negative}
\end{figure}

\begin{rem}
Stoimenow \cite[Corollary~$5$]{stoimenow1} improved the Hirasawa's result \cite{hirasawa1} which states if a canonical Seifert surface of an almost alternating diagram is compressible, the diagram has a  ``d-cycle", that is,  there is a crossing joining the same two Seifert circles as the circles which the dealternator connects. 
\end{rem}
%
\begin{rem}
Let $D$ be an almost positive link diagram. 
If there is a crossing joining the same two Seifert circles of $D$ as the  two circles which the negative crossing connects, the diagram $D$ is a {\it quasipositive diagram} introduced by Baader \cite{baader2}. 
In particular, $D$ represents a quasipositive link. 
The author expects that such an almost positive diagram $D$ represents a positive link. 
If so, we obtain that any almost positive diagram of an almost positive link has the minimal genus. 
\end{rem}
%
\par
Next, we compute the Rasmussen invariant of an oriented knot which has an almost positive diagram (such a knot is positive or almost positive). 
The Rasmussen invariant of a knot is closely related to the $0$-th term of the Khovanov homology. 
The following lemma is beneficial when we compute the Rasmussen invariant. 
\begin{lem}\label{lem1}
Let $D$ be an almost positive link diagram of a non-split link $L$ with a negative crossing $p$. 
If there is no (positive) crossing of $D$ joining the same two Seifert circles as the two circles which are connected by the negative crossing $p$, we have $\KH^{0, 2g_{3}(D)+\sharp L-4}(L)=0$, where $\sharp L$ is the number of the components of $L$. 
\end{lem}
\begin{proof}
Let $n$ be the number of crossings of $D$ and $s$ the number of Seifert circles of $D$. By the definition, we obtain 
\begin{align}
\KH^{0, 2g_{3}(D)+\sharp L-4}(L)
&=H^{1, 2g_{3}(D)+\sharp L-1-3-(n-1)+2}(D) \label{compute1}\\
&=H^{1, n-s+1-3-(n-1)+2}(D) \nonumber\\
&=H^{1, -s+1}(D). \nonumber
\end{align}
Let us prove $H^{1, -s+1}(D)=0$. 
Order the crossings of $D$ so that $p$ is the first crossing. 
Define $\varepsilon^{(j)}_{j}:=1$ and $\varepsilon^{(j)}_{i}:=0$ for $i\neq j$. 
Put $\varepsilon^{(j)}:=(\varepsilon^{(j)}_{1}, \varepsilon^{(j)}_{2}, \dots, \varepsilon^{(j)}_{n-1}, \varepsilon^{(j)}_{n})\in \{0,1\}^{n}$ (that is, the elements of $\varepsilon^{(j)}$ are $0$ except the $j$-th element which is $1$). 
Note that the smoothing $D_{\varepsilon^{(1)}}$ is the Seifert smoothing. 
Since there is no crossing of $D$ joining the same two Seifert circles as the two circles which are connected by $p$, the number $k_{\varepsilon^{(j)}}$ of circles of the smoothing $D_{\varepsilon^{(j)}}$ is given as follows: 
\begin{align}
k_{\varepsilon^{(j)}}=
\begin{cases}
s& \text{if}\ j=1, \\
s-2& \text{if}\ j\neq 1.  
\end{cases} \label{no_circle}
\end{align}
From the definition of $H^{1, -s+1}(D)$ and $(\ref{no_circle})$, we have 
\begin{align}
&H^{1, -s+1}(D) \label{compute2}\\
&=\ker d^{1}\cap \{v\in C^{1}(D) \mid deg (v)=-s\}/d^{0}(\{ w \in C^{0}(D) \mid deg (w)=-s+1\}) \nonumber\\
&=\mathbf{Q}\{x^{\otimes s} \in M_{\varepsilon^{(1)}}\}/d^{0}(\mathbf{Q}\{x^{\otimes s-1} \in C^{0}(D)\}). \nonumber
\end{align}
From $(\ref{no_circle})$, we compute 
\begin{align}
d^{0}(x^{\otimes s-1})&=(d_{\mathbf{0}\rightarrow \varepsilon^{(1)}}(x^{\otimes s-1}), d_{\mathbf{0}\rightarrow \varepsilon^{(2)}}(x^{\otimes s-1}), \dots, d_{\mathbf{0}\rightarrow \varepsilon^{(n)}}(x^{\otimes s-1})) \label{compute3}\\
&=(\Delta (x)\otimes x^{\otimes s-2}, 0, \dots, 0) \nonumber \\
&=(x^{\otimes s}, 0, \dots, 0) \in \bigoplus_{j=1}^{n}M_{\varepsilon^{(j)}} (=C^{1}(D)), \nonumber 
\end{align}
(we obtain the second equality because the map $d_{\mathbf{0}\rightarrow \varepsilon^{(1)}}$ is given by the comultiplication map $\Delta$, the maps $d_{\mathbf{0}\rightarrow \varepsilon^{(2)}}, \dots, d_{\mathbf{0}\rightarrow \varepsilon^{(n)}}$ are given by the multiplication maps $m$ and we have $m(x\otimes x)=0$). \par
Hence, by $(\ref{compute1})$, $(\ref{compute2})$ and $(\ref{compute3})$, we obtain 
\begin{align*}
\KH^{0, 2g_{3}(D)+\sharp L-4}(L)&=H^{1, -s+1}(D)\\
&=\mathbf{Q}\{x^{\otimes s} \in M_{\varepsilon^{1}}\}/d^{0}(\mathbf{Q}\{x^{\otimes s-1} \in C^{0}(D)\})\\
&=\mathbf{Q}\{x^{\otimes s}\}/\mathbf{Q}\{x^{\otimes s}\}\\
&=0. 
\end{align*}
\end{proof}
The two results, Theorem~$\ref{stoimenow1}$ and Lemma~$\ref{lem1}$, allow us to prove Theorem~$\ref{main2}$ below. 

\begin{thm}\label{main2}
Let $D$ be an almost positive diagram of a knot $K$ with negative crossing $p$. 
Denote the Rasmussen invariant of $K$ by $s(K)$ and the genus of the Seifert surface obtained from $D$ (by Seifert's algorithm) by $g_{3}(D)$. 
\begin{enumerate}
\item If there is no (positive) crossing joining the same two Seifert circles of $D$ as the two circles which are connected by the negative crossing $p$, we obtain $s(K)=2g_{3}(D)$.  
\item If there is a (positive) crossing joining the same two Seifert circles of $D$ as the two circles which are connected by the negative crossing $p$, we obtain $s(K)=2g_{3}(D)-2$. 
\end{enumerate}
\end{thm}
\begin{proof}
Let $D_{+}$ be the positive diagram obtained from $D$ by crossing change at $p$ and $K_{+}$ the knot represented by $D_{+}$. 
It is known (see \cite[Corollary~$4.3$, Theorem~$1$ and Theorem~$4$]{rasmussen1}) that we have 
\begin{align}
s(K_{+})-2&\leq s(K)\leq s(K_{+}), \label{1}\\
|s(K)|&\leq 2g_{4}(K)\leq 2g_{3}(K)  \label{2}, \\
s(K_{+})=2g_{4}(K_{+})&=2g_{3}(K_{+})=2g_{3}(D_{+})(=2g_{3}(D)). \label{3}
\end{align}
\par
$(1)$ In the case where there is no (positive) crossing joining the same two Seifert circles as the two circles which are connected by $p$: 
By $(\ref{1})$, we can easily see that $s(K)=s(K_{+})$ or $s(K_{+})-2$ (since $s(K)$ is an even integer for any knot $K$). 
From Lemma~$\ref{lem1}$, Proposition~$\ref{key}$ and $(\ref{3})$, we have $s(K)\neq 2g_{3}(D)-2=s(K_{+})-2$. Hence, we obtain $s(K)=s(K_{+})=2g_{3}(D)$. 
\par
$(2)$ In the case where there is a (positive) crossing joining the same two Seifert circles as the two circles which are connected by $p$: From Theorem~$\ref{stoimenow1}$, $(\ref{2})$ and $(\ref{3})$, we obtain 
\begin{align*}
2g_{3}(D)-2=s(K_{+})-2\leq s(K)\leq 2g_{4}(K)\leq 2g_{3}(K)= 2g_{3}(D)-2. 
\end{align*}
\end{proof}
\begin{proof}[Proof of Theorem~$\ref{main}$]
This immediately follows from Theorems~$\ref{stoimenow1}$ and $\ref{main2}$ and $(\ref{2})$. 
\end{proof}
%
\begin{proof}[Proof of Corollary~$\ref{cor1}$]
Abe \cite[Proof of Theorem~$1.3$]{abe2} proved that a homogeneous knot $K$ satisfying $s(K)=2g_{4}(K)=2g_{3}(K)$ is a positive knot. 
From Theorem~$\ref{main}$, any almost positive knot is not homogeneous.  
\end{proof}
\begin{proof}[Proof of Corollary~$\ref{cor2}$]
Stoimenow \cite[Theorem~$7.1$]{stoimenow2} proved that there is no almost positive knot of $3$-genus one. 
From Theorem~$\ref{main}$, the $4$-genus of any almost positive knot is not one.  
\end{proof}
\begin{rem}
It is natural to consider the following question: ``Are Theorems~$\ref{main}$ and $\ref{main2}$ true for any almost positive link?". 
An answer to the question will be given in \cite{abe-tagami}, where we use the Rasmussen invariant extended to links by Beliakova and Wehrli \cite{cat_colored}. 
\end{rem}
%
%
%
%
%
%
%
%
%
%
\noindent{\bf Acknowledgements: } 
The author is grateful to the referee for his/her comments. 
This work was supported by JSPS KAKENHI Grant number 13J01362. 
%
\bibliographystyle{amsplain}
\bibliography{tagami}
\end{document}